\documentclass[eqno,a4paper,11pt]{article}
\title{Existence and nonexistence theorems for global weak solutions to  quasilinear wave equations  for the elasticity }
 \author{Yun-guang  Lu\footnote{K.K.Chen Institute for Advanced Studies\
Hangzhou Normal University, P. R. CHINA\
ylu2005@ustc.edu.cn} and Yuusuke Sugiyama\footnote{Department of Mathematics,\ Tokyo University of Science\ 
Kagurazaka 1-3, Shinjuku-ku, Tokyo 162-8601, Japane-mail:sugiyama@ma.kagu.tus.ac.jp}
}
\date{}
\usepackage{amssymb,amsmath,amsthm,mathrsfs,delarray,enumerate}
\usepackage{graphics}

\theoremstyle{definition} 
\newtheorem{Def}{Deffinition}[section]

\newtheorem{Lemma}[Def]{Lemma}

\newcommand{\R}{\mathbb{R}} 

 \newcommand{\F}[2]{\frac{#1}{#2}}
 
\newcommand{\E}[0]{\varepsilon}
\newcommand{\lam}[0]{\lambda}
 \newcommand{\BE}[1]{\begin{array}{#1}}
\newcommand{\EE}[0]{\end{array}}
\makeatother

\newcommand{\usection}[1]{section}
\newtheorem{theorem}{\underline{Theorem}}
\newtheorem{lemma}[theorem]{\underline{Lemma}}
\newtheorem{Remark}[theorem]{\underline{Remark}}

\makeatletter

\@addtoreset{equation}{section}

\begin{document}
\maketitle %
\begin{abstract} In this paper,  by using
the theory of compensated compactness coupled with the
kinetic formulation by Lions, Perthame, Souganidis and Tadmor \cite{LPT,LPS},
 we prove the existence and nonexistence of  global generalized (nonnegative) solutions  of the nonlinearly degenerate wave equations
$v_{tt} =c (|v|^{s-1} v)_{xx}$ with the nonnegative initial
data $v_{0}(x)$ and $ s > 1$. This result is an extension of the results in the second author's paper \cite{Su}, where the existence and the
nonexistence of the unique global classical solution
were studied with a threshold on $\int_{-\infty}^{\infty}  v_{1}(x) dx$ and the non-degeneracy condition
$ v_{0}(x) \geq c_{0} > 0$ on the initial data.
\end{abstract}

 { Key Words: Global weak solutions; degenerate wave equations; viscosity method;
 compensated compactness \\}
\section{ \bf Introduction}
In this paper, we study the global generalized solutions  of the nonlinearly degenerate wave equations
\begin{equation}
\label{1.1} v_{tt}=c(|v|^{s-1}v)_{xx}, \quad - \infty < x <
\infty, \quad t>0,
\end{equation}
 with the initial data
\begin{equation}
\label{1.2} (v,v_{t})|_{t=0}=(v_{0}(x),v_{1}(x)), \quad - \infty <
x < \infty.
\end{equation}
To prove  the global  existence of solutions to \eqref{1.1}, we also consider the nonlinear hyperbolic systems of elasticity
 \begin{equation}
\label{1.3} v_t-u_x=0, \quad u_{t}-c(|v|^{s-1}v)_{x}=0
\end{equation}
with bounded initial date
\begin{equation}
\label{1.4} (v,u)|_{t=0}=(v_{0}(x),u_{0}(x)),
\end{equation}
where $v_{0}(x) \ge 0, s>1, c= \F{ \theta^{2}}{s}>0$  and $ \theta= \F{s+1}{2} > 0$ are constants.

A function $v \in L^{\infty} (\R \times \R^+)$ is called a generalized
solution of the Cauchy problem (\ref{1.1})-(\ref{1.2})
 if for any test function $\phi \in C_{0}^{\infty}( \R \times [0,\infty))$,
\begin{equation}
\label{1.5}
      \int_{0}^{\infty} \int_{-\infty}^{\infty} v \phi(x,t)_{tt}-c|v|^{s-1}v \phi(x,t)_{xx}dxdt
      + \int_{-\infty}^{\infty} v_{0}(x)\phi(x,0)_{t}- v_{1}(x) \phi(x,0) dx
      =0.
\end{equation}
A pair of functions $(v(x,t),u(x,t)) \in L^{\infty} (\R \times \R^+) \times L^{\infty} (\R \times [0,\infty))$ is called a generalized
 solution of the Cauchy problem (\ref{1.3})-(\ref{1.4}) if for any test function $\phi_{i}(x,t) \in C_{0}^{\infty}(\R
\times [0,\infty))$, i=1,2,
\begin{equation}\left\{
\label{1.6}
\begin{array}{l}
      \int_{0}^{\infty} \int_{-\infty}^{\infty} v \phi(x,t)_{1t}-u \phi(x,t)_{1x}dxdt
      + \int_{-\infty}^{\infty} v_{0}(x) \phi_{1}(x,0) dx =0,   \\  \\
       \int_{0}^{\infty} \int_{-\infty}^{\infty} u \phi(x,t)_{2t}-c|v|^{s-1}v \phi(x,t)_{2x}
       dxdt + \int_{-\infty}^{\infty} u_{0}(x) \phi_{2}(x,0) dx=0.
\end{array}\right.
\end{equation}
It is obvious that the generalized solutions of the Cauchy problem
(\ref{1.3})-(\ref{1.4}) are also the solutions of the Cauchy
problem (\ref{1.1})-(\ref{1.2}) if we specially choose
$\phi_{1}=\phi_{t}, \phi_{2}=\phi_{x}$ and $v_{1}(x)=u_{0}'(x)$  in
(\ref{1.6}).

Two eigenvalues of (\ref{1.3}) are
\begin{equation}
\label{1.7}
\lam_1=- (cs)^{\F{1}{2}}|v|^{\F{s-1}{2}}= - \theta |v|^{\F{s-1}{2}}, \hspace{0.3cm}
\lam_2= (cs)^{\F{1}{2}}|v|^{\F{s-1}{2}}= \theta |v|^{\F{s-1}{2}},
      \end{equation}
which coincide when  $ v=0$ and so in which \eqref{1.3} is nonstrictly
hyperbolic; two Riemann invariants of
\eqref{1.3} are
\begin{equation*}
z=u+ \int_{0}^{v} (cs)^{\F{1}{2}}|v|^{\F{s-1}{2}}  dv , \hspace{0.3cm}
w=u- \int_{0}^{v} (cs)^{\F{1}{2}}|v|^{\F{s-1}{2}} dv.
      \end{equation*}
Throughout this paper, we concentrate our study on the domain of $v \ge 0 $, then
the eigenvalues of (\ref{1.3}) are $ \lam_1=- \theta  v^{\F{s-1}{2}}, \lam_2=   \theta v^{\F{s-1}{2}}$
and the Riemann invariants are
\begin{equation*}
z=u+  \F{ 2(cs)^{\F{1}{2}}}{s+1} v^{\F{s+1}{2}}=  u+ v^{ \theta} , \hspace{0.3cm}
w=u- \F{ 2(cs)^{\F{1}{2}}}{s+1} v^{\F{s+1}{2}}=u- v^{ \theta}.
\end{equation*}
The Riemann invariants for \eqref{1.1} are given formally by
\begin{equation}
\label{1.9-2}
w=u -v^\theta, \hspace{0.3cm} z=\tilde{u} +v^\theta,
\end{equation}
where 
\begin{equation*}
u=\int_{-\infty} ^x v_t dx, \hspace{0.3cm} \tilde{u}=-\int_x ^\infty v_t dx.
\end{equation*}
We denote $w(x,0)$ and $z(x,0)$ by $w_0 (x)$ and $z_0 (x)$ respectively.
We have the following first main result in this paper.
\begin{theorem}\label{main1} (I). Suppose that
$z_0(x)$ and $w_0(x)) $  are decreasing and satisfy
\begin{equation}
\label{1.10}
  c_1 \leq w_0(x) \leq c_0  \leq z_0(x) \leq c_2,
  \end{equation}
where $c_1,c_0,c_2$ are three constants.  Then the Cauchy problem (\ref{1.3}) and (\ref{1.4}) has a  global weak
solution satisfying (\ref{1.6}).

(II). Let $v_0(x) \geq 0$ be bounded,  $ v_1(x) \in L^{1}(\R) $ and the limits $v_{0}^{\theta}(x)|_{x=\pm  \infty}$ exist. Moreover suppose that
\begin{equation}
\label{1.11}
  \int_{-\infty}^{\infty}  v_{1}(x) dx +
  v_{0}^{\theta}(x)|_{x=+ \infty} + v_{0}^{\theta}(x)|_{x=- \infty} \geq 0
  \end{equation}
and
\begin{equation}
\label{1.12}
  v_1(x) \pm    \theta v_{0}^{\F{s-1}{2}}(x) \partial_{x} v_{0}(x)  \leq 0.
 \end{equation}
Then the Cauchy problem \eqref{1.1} and \eqref{1.2} has a  global weak
solution satisfying \eqref{1.5}.
\end{theorem}
Next we give the second main theorem of this paper, which implies the nonexistence of solutions  satisfying $v \geq 0$.
From the proof of Theorem \ref{main1}, we can easily check that the global solutions of \eqref{1.1} and \eqref{1.3} constructed Theorem \ref{main1} satisfy that
 \begin{eqnarray}
\label{con-1}
v(x,t) \geq 0  \ \mbox{for a.a.} \ (x,t) \in \R \times \R^{+}
\end{eqnarray} 
\begin{eqnarray}
\label{con0}
\mbox{$w(x,t)$  and $z(x,t)$ are decreasing with $x$ for a.a. $t \geq 0$},
\end{eqnarray} 
\begin{eqnarray} 
\label{con--1}
w, z\in  L^{\infty} (\R \times \R^{+}).
\end{eqnarray}
Furthermore, we can show that the following properties are also satisfied for the global solutions in Theorem \ref{main1}, if we assume the additional regularity on initial data that $w_0, z_0 \in W^{1,1} _{loc} (\R)$:
\begin{eqnarray} 
\label{con--1-2}
w, z\in   W^{1,1} _{loc} (\R \times \R^{+}),
\end{eqnarray}
\begin{eqnarray}
\label{con2}
\| w_x (t)\|_{L^1}+ \| z_x (t) \|_{L^1} \leq \| w_x (0)\|_{L^1} +  \| z_x (0)\|_{L^1} \ \mbox{for a.a.} \ t\geq 0 
\end{eqnarray}
and
\begin{eqnarray}
\label{con3}\\
\| w_t (t)\|_{L^1}+ \| z_t (t) \|_{L^1} \leq C (\| w_x (0)\|_{L^1} +  \| z_x (0)\|_{L^1}) \ \mbox{for a.a.} \ t\geq 0.\notag
\end{eqnarray}
Furthermore,  the global solution of \eqref{1.1}  satisfies that
\begin{eqnarray}\label{con4}
\partial_t v(t,\cdot ) \in L^1 (\R)  \ \mbox{for a.a.} \ t \geq 0.
\end{eqnarray}
We have the following second main result in this paper.
\begin{theorem}\label{main2}(I).
Suppose that $w_0, z_0 \in W^{1,1} _{loc} (\R) \cap L^\infty (\R)$ and that $w_0$ and $z_0$ are  decreasing. 
Furthermore, we assume that for some $x, y \in \R$
\begin{eqnarray} \label{xy}
w_0 (x) > z_0 (y).
\end{eqnarray}
Then the Cauchy problem \eqref{1.3} and \eqref{1.4}  has no solutions satisfying the properties  \eqref{con-1}-\eqref{con3}.

(II). Let $v_0 \geq 0$, $w_0, z_0 \in W^{1,1} _{loc} (\R) \cap L^\infty (\R)$ and $v_1 \in L^1 (\R)$. Suppose  that the limits $v_{0}^{\theta}(x)|_{x=\pm  \infty}$ exist.
 Furthermore, we assume that \eqref{1.12} is satisfied and 
\begin{eqnarray} \label{xy2}
\int_{-\infty} ^{\infty} v_1 (x) dx +v_0 ^\theta (x)|_{x=\infty} +v_0 ^\theta (x)|_{x=-\infty} <0 .
\end{eqnarray}
Then the Cauchy problem \eqref{1.1} and \eqref{1.2}  has no solutions  satisfying the properties  \eqref{con-1}-\eqref{con4}.
\end{theorem}

This paper extends the results in \cite{Su}. In \cite{Su}, the second author obtained a threshold $- \F{2}{a+1}$ of
$\int_{-\infty}^{\infty}  v_{1}(x) dx $ for the equation $v_{tt}=((1+v)^{2 \alpha} v_{x} )_{x} $ . If the initial data \eqref{1.2} satisfy
\begin{equation}
\label{1.13}
  v_1(x) \pm    (1+ v_{0})^{\alpha}(x) \partial_{x} v_{0}(x)  \leq 0
  \end{equation}
  and
\begin{equation*}
(v_{0}, v_{1}) \in H^{2}(\R) \times H^{1}(\R),  \quad 1+ v_{0}(x) \geq c_{0} >0,  \quad \int_{-\infty}^{\infty}  v_{1}(x) dx > - \F{2}{\alpha +1},
  \end{equation*}
  then the Cauchy problem (\ref{1.1}) and (\ref{1.2}) has a global unique solution $ v \in C( [0, \infty); H^{2}(\R)) \cap  C^{1}( [0, \infty); H^{1}(\R)) $
  satisfying $ v_{0}(x) \geq c_{1} >0 $ for all $(x,t) \in [0, \infty) \times \R$, where $ c_{0}, c_{1}$ are constants (The proof is classical and can be found
  in \cite{J, YN}). If the initial data (\ref{1.2}) satisfy (\ref{1.13}) and
  \begin{equation*}
(v_{0}, v_{1}) \in H^{2}(\R) \times H^{1}(\R),  \quad 1+ v_{0}(x) \geq c_{0} >0,  \quad \int_{-\infty}^{\infty}  v_{1}(x) dx < - \F{2}{\alpha+1},
  \end{equation*}
  then equation (\ref{1.1}) must degenerate at a finite time, namely, there exists $T^{*} >0$ such that a local unique solution
  $ v \in C( [0, T^{*}); H^{2}(\R)) \cap  C^{1}( [0, T^{*}); H^{1}(\R)) $ of the Cauchy problem  (\ref{1.1}) and (\ref{1.2}) exists, and
   \begin{equation*}
\lim_{ t \nearrow T^{*}}  1+ v_{0}(t, x_{0})=0 \quad    \mbox{ for some  }  \quad x_{0} \in \R.
  \end{equation*}
 One can expect that classical solutions can not be extended after the degenerate occurs, since
the strict hyperbolicity is lost. To avoid this difficulty, we treat our problem in the framework of the weak solution and
construct a solutions with $ v_{0}(x)$ having the degeneracy ($v_{0}(x_{0})=0$ for some $x_{0} \in \R$).
Under the assumption that $v \geq 0$, the decreasing property of $w_0$ and $z_0$ ensures the absence of shock waves.
Our main theorems give a threshold separating the existence and nonexistence of solutions to \eqref{1.1}
satisfying $v\geq 0$, under the assumption that  $w_0$ and $z_0$ are decreasing.
In \cite{ZZ1, ZZ2}, Zhang and Zheng  proved  the global existence of weak solutions to the 1D variational wave equation $v_{tt}=c(v)(c(v)v_x)_x$
with  the non-degeneracy condition that $c(u) \geq  c_0 $ for some constant $c_0 >0$.

This paper is organized as follows: In Section 2, we introduce a variant of the viscosity argument, and construct approximated solutions
of the Cauchy problem (\ref{1.3}) and (\ref{1.4}) by using the solutions of the parabolic system (\ref{2.1}) with the initial data (\ref{2.2}).
Under the conditions in the Part I of Theorem 1, we can easily obtain the necessary boundedness estimates (\ref{2.6}) and (\ref{2.7}) of
$(w^{\E, \delta}(x,t), z^{\E, \delta}(x,t))$, where the bound $M(\delta)$ in (\ref{2.7}) could tend to infinity  as $\delta$ goes to zero. Based on the
estimates (\ref{2.6}) and (\ref{2.7}) and the kinetic formulation by Lions, Perthame, Souganidis and Tadmor \cite{LPT,LPS}, in Section 3,
we prove the pointwise convergence
of $(u^{\E, \delta}(x,t), v^{\E, \delta}(x,t))$ by using the theory of the compensated compactness, and its limit $(u,v)$ is a generalized solution  of the Cauchy problem (\ref{1.3}) and (\ref{1.4}).
Finally, in the last part of Section 3, we shall prove all conditions in the Part I of Theorem \ref{main1} are satisfied under the assumptions of the
Part II of Theorem \ref{main1}, and obtain the global existence of the  generalized solutions  of the Cauchy problem (\ref{1.1}) and (\ref{1.2}),
which completes the proof of Theorem \ref{main1}. In Section 4, we prove Theorem \ref{main2}.
The proof of Theorem \ref{main2} are based on similar, but refined and simpler method in \cite{Su}
The key for the proof in \cite{Su} is the function $F(t)$:
$$
F(t)=\int_{-\infty} ^\infty  v(x,t) - v(x,0)dx.
$$
In the estimate for $F(t)$, we divide the integral region $(-\infty,\infty)$ of $F$ by three parts, using
characteristic curves. However, for non-classical solutions, the characteristic curves would not be defined.
Observing $0 \leq w_t (x,t)$ and $z_t (x,t) \leq 0$, we estimate $F$ more simply.

\section{ \bf Viscosity Solutions}

In this section we construct the approximated solutions of the Cauchy problem (\ref{1.3}) and (\ref{1.4}) by using the following
parabolic systems
\begin{equation}\left\{
\begin{array}{l}
\label{2.1}
      w_{t}+ \lam_2 w_{x}=\E w_{xx}         \\\\
       z_{t}+ \lam_1 z_{x}=\E z_{xx}
\end{array}\right.
\end{equation}
with initial data
\begin{equation}
\label{2.2}
  (w(x,0),z(x,0))=(w^{\delta}_{0}(x),z^{\delta}_{0}(x)) \ast G^{\delta},
\end{equation}
where $\E, \delta $ are small positive constants, $G^{\delta}$ is a mollifier,
\begin{equation*}\left\{
\begin{array}{l}
w^{\delta}_{0}(x)= u_{0}(x)-   (v^{\delta}_{0}(x))^{ \theta} =
 u_{0}(x)-    (v_{0}(x))^{ \theta} - \delta   \\\\
 z^{\delta}_{0}(x)= u_{0}(x)+  (v^{\delta}_{0}(x))^{ \theta} =
 u_{0}(x)+     (v_{0}(x))^{ \theta} + \delta
      \end{array}\right.
\end{equation*}
and $(u_{0}(x),v_{0}(x))$ are given by (\ref{1.4}). Thus $ w(x,0)$ and $ z(x,0)$ are smooth functions, and satisfy
\begin{equation}
\label{2.4}
c_1- \delta \leq w(x,0) \leq c_{0}-  \delta , \hspace{0.3cm}
c_{0} + \delta   \leq z(x,0) \leq c_2 +  \delta,
\end{equation}
\begin{equation}
\label{2.5}
  -M(\delta) \leq w_{x}(x,0) \leq 0, \hspace{0.3cm}  -M(\delta) \leq z_{x}(x,0)  \leq 0,
\end{equation}
where $ M(\delta)$ is a constant, and could tend to  infinity as $\delta$ tends to zero.

First, we have the following Lemma.
\begin{lemma}\label{l1} Let  $w(x,0),z(x,0)$
be bounded in $ C^1$ space and satisfy (\ref{2.4}) and (\ref{2.5}).
Moreover, suppose that $(w^{\E, \delta}(x,t), z^{\E, \delta}(x,t))$ is a smooth solution of (\ref{2.1}), (\ref{2.2}) defined in a strip $(- \infty, \infty) \times [0,T]$ with
$0 < T < \infty$. Then
 \begin{equation}
 \label{2.6}
c_1- \delta \leq w^{\E, \delta}(x,t) \leq c_{0} - \delta ,\hspace{0.2cm}  c_{0} + \delta \leq z^{\E, \delta}(x,t) \leq c_2+ \delta ,
\end{equation}
and
\begin{equation}
\label{2.7}
-M(\delta) \leq w^{\E, \delta}_{x}(x,t) \leq 0, \hspace{0.3cm}  -M(\delta) \leq z^{\E, \delta}_{x}(x,t)  \leq 0.
\end{equation}
\end{lemma}
\begin{proof}
 The estimates in (\ref{2.6})  can be obtained by using the maximum principle to (\ref{2.1}), (\ref{2.2}) and the condition (\ref{2.4}) directly.

We differentiate (\ref{2.1}) with respect to $x$ and let $w_x=-R, z_x=-S$; then
\begin{equation}\left\{
\begin{array}{l}
\label{2.8}
      R_{t}+ \lam_2 R_{x}- ( \lam_{2w} R+ \lam_{2z}S)R= \E R_{xx},         \\\\
      S_{t}+ \lam_1 S_{x}- ( \lam_{1w}R+ \lam_{1z}S)S= \E S_{xx}.
\end{array}\right.
\end{equation}
The nonnegativity $w^{\E, \delta}_{x}(x,t) \leq 0$  and $z^{\E, \delta}_{x}(x,t)  \leq 0$  in (\ref{2.7}) can be obtained by using the maximum principle
 to (\ref{2.8}) and the condition $ w_{x}(x,0) \leq 0$ and  $z_{x}(x,0)  \leq 0 $ in (\ref{2.5}).

A simple calculation yields
\begin{equation*}
u_{w}= \F{1}{2}, \hspace{0.3cm} u_{z}= \F{1}{2} , \hspace{0.3cm} v_{w}=-  \F{1}{2  \theta v^{\F{s-1}{2}} } , \hspace{0.3cm}
v_{z}= \F{1}{2  \theta v^{\F{s-1}{2}} }
\end{equation*}
and
\begin{equation}
\lam_{1w}= \lam_{2z}= \F{s-1}{4v} > 0, \hspace{0.3cm} \lam_{1z}= \lam_{2w}=- \F{s-1}{4v} < 0.
\label{2.10}
\end{equation}
Thus the lower bound $ -M(\delta) \leq w^{\E, \delta}_{x}(x,t),  -M(\delta) \leq z^{\E, \delta}_{x}(x,t) $  in (\ref{2.7})  is a direct conclusion of Lemma 2.4
in \cite{Lu1}. Lemma \ref{l1} is proved.
\end{proof}

From the estimates  in (\ref{2.6}), we have the following estimates about the functions
$u^{\E, \delta}(x,t)$ and $v^{\E, \delta}(x,t)$,
\begin{equation}
\label{2.11}
\F{c_{1}+c_{0}}{2}  \leq u^{\E, \delta}(x,t) \leq \F{c_{2}+c_{0}}{2},  \hspace{0.3cm}   c_{1}-c_{0}+ 2 \delta  \leq  2 (v^{\E, \delta}(x,t))^{ \theta} \leq c_{2}-c_{1}+ 2 \delta
\end{equation}
and
\begin{equation}
\label{2.12}
-M(\delta)   \leq u^{\E, \delta}_{x}(x,t) \leq 0,  \hspace{0.3cm}   -M(\delta)   \leq 2 \theta (v^{\E, \delta}(x,t))^{\F{s-1}{2}}  v^{\E, \delta}_{x}(x,t) \leq M(\delta)
\end{equation}
or
\begin{equation}
\label{2.13}
| v^{\E, \delta}_{x}(x,t)| \leq M_{1}(\delta),
\end{equation}
where $M_{1}(\delta)$  is a constant, and could tend to  infinity as $\delta$ tends to zero.

The positive, lower bound $ 2 (v^{\E, \delta}(x,t))^{ \theta} \geq c_{1}-c_{0}+ 2 \delta  $  in (\ref{2.11}) ensures the regularity of $\lam_{iw}, \lam_{2z}, i=1,2$ from (\ref{2.10}) and the following global existence of the Cauchy problem (\ref{2.1})-(\ref{2.2}).
\begin{theorem}\label{theo}Let  $(w(x,0), z(x,0)) \in C^1 $ satisfy \eqref{2.4} and \eqref{2.5}, then the Cauchy problem \eqref{2.1} with initial data \eqref{2.2} has a unique global smooth
solution satisfying \eqref{2.6} and \eqref{2.7}.
\end{theorem}
The proof of Theorem \ref{theo} is standard, and the details can be found in \cite{Sm} or Theorem 2.3 in \cite{Lu1}.

\section{ \bf Proof of Theorem 1.}
In this section, we prove Theorem 1.

A pair of smooth functions $( \eta(v,u), q(v,u))$ is called a pair
of entropy-entropy flux of system (\ref{1.3}) in the region of $ v > 0$ if $( \eta(v,u),
q(v,u))$ satisfies
\begin{equation}
\label{3.1}
q_{u}=-\eta_{v},  \quad q_{v}=- \theta^{2} v^{s-1}
\eta_{u}.
\end{equation}
Eliminating the $q$ from (\ref{3.1}), we have the following
entropy equation of system (\ref{1.3}).
\begin{equation}
\label{3.2} \eta_{vv}= \theta^{2}v^{s-1}  \eta_{uu}.
\end{equation}
Consider (\ref{3.2}) in the region $v > 0$ with the following
initial conditions
\begin{equation}
\label{3.3} \eta(0,u)=d^{0}f(u),   \quad
\eta_{v}(0,u)=0,
\end{equation}
where $d^{0}= \int_{-1}^{1} (1- \tau^{2})^{ \lam} d \tau.$
Then from the results in \cite{JPP,LPS,LPT}, an entropy of (\ref{3.2}) with (\ref{3.3}) in the region
$v>0$ is
\begin{equation*}
 \eta^{0}(v,u)=\int_{-\infty}^{\infty}
f(\xi)G(v,\xi-u) d \xi,
\end{equation*}
where the fundamental solution
\begin{equation*}
 G(v,u-\xi)=v(v^{s+1}-(\xi-u)^{2})_{+}^{ \lam},
\end{equation*}
the notation $x_{+}=max (0,x)$ and $ \lam=- \F{s+3}{2(s+1)} \in
(-1,0)$. The entropy flux $q^{0}(v,u)$ associated with
$\eta^{0}(v,u)$ in the region $v>0$ is
\begin{equation*}
 q^{0}(v,u)= - \int_{-\infty}^{\infty}
f(\xi) \theta \F{\xi-u}{v} G(v,\xi-u) d \xi.
\end{equation*}

Letting $ \xi=u+v^{\F{s+1}{2}} \tau$, by simple calculations, we have on
$v>0$
\begin{equation*}\begin{array}{ll}
 \eta^{0}(v,u)  = \int_{-\infty}^{\infty}
f(\xi)G(v,\xi-u) d \xi = \int_{w}^{z} f(\xi)v (z-
\xi)^{\lam}(\xi-w)^{\lam} d \xi \\  &\\ = \int_{-1}^{1}
f(u+v^{\F{s+1}{2}} \tau)v(v^{\F{s+1}{2}})^{(2 \lam+1)} (1-
\tau^{2})^{\lam} d \tau = \int_{-1}^{1} f(u+v^{\F{s+1}{2}}
\tau) (1- \tau^{2})^{\lam} d \tau,
\end{array}
\end{equation*}
and
\begin{equation*}\begin{array}{ll}
 q^{0}(v,u)  = - \int_{-\infty}^{\infty}
f(\xi) \theta \F{\xi-u}{v} G(v,\xi-u) d \xi = -
\int_{w}^{z} f(\xi) \theta (\xi-u)(z- \xi)^{\lam}(\xi-w)^{\lam} d \xi \\
&\\ = - \int_{-1}^{1} f(u+v^{\F{s+1}{2}} \tau) \theta
 v^{\F{s+1}{2}} \tau (v^{\F{s+1}{2}})^{(2 \lam+1)} (1- \tau^{2})^{\lam} d
 \tau \\
 &\\=-
\int_{-1}^{1} f(u+v^{\F{s+1}{2}} \tau) \theta v^{\F{s-1}{2}}
\tau (1- \tau^{2})^{\lam} d \tau.
\end{array}
\end{equation*}

We consider the matrix
 \begin{equation*}
A =\left( \begin{array}{cc}
a & b\\\\
e & d
\end{array} \right)= \left( \begin{array}{cc}
w_{v} & w_{u}\\\\
z_{v} & z_{u}
\end{array} \right)^{-1}=\left( \begin{array}{cc}
- \F{1}{2 \theta  v^{\F{s-1}{2}} } & \F{1}{2 \theta  v^{\F{s-1}{2}} }\\\\
 \F{1}{2} & \F{1}{2}
\end{array} \right),
\end{equation*}
and multiply it by the two sides in (\ref{2.1}), then (\ref{2.1}) can be rewritten
as follows:
\begin{equation}\begin{array}{ll}
\label{3.10}
      v_t-u_x=  \E (a w_{xx}+ bz_{xx}) =   \\\\
       \E (a w_{x}+ bz_{x})_{x}-  \E (a_{x} w_{x}+ b_{x}z_{x})
      = \E v_{xx}-   \E (a_{x} w_{x}+ b_{x}z_{x})
\end{array}
\end{equation}
and
   \begin{equation}\begin{array}{ll}
\label{3.11}
       u_{t}-c(v^{s})_{x}= \E (e w_{xx}+ d z_{xx}) \\\\
       = \E (e w_{x}+ dz_{x})_{x}-  \E (e_{x} w_{x}+ d_{x}z_{x})
       = \E u_{xx}.
\end{array}
\end{equation}

First, we have the following Lemma:
\begin{lemma} \label{l2}
\begin{equation}
 \label{3.12}
v^{\E, \delta}(x,t)_{t}-u^{\E,\delta}(x,t)_{x},
\end{equation}
\begin{equation}
 \label{3.13}
 u^{\E, \delta}_{t}-c((v^{\E, \delta})^{s})_{x}
 \end{equation}
 and
\begin{equation}
\label{3.14}
\eta^{0}(v^{\E,\delta}(x,t),u^{\E,\delta}(x,t))_{t}
+q^{0}(v^{\E,\delta}(x,t),u^{\E,\delta}(x,t))_{x}
\end{equation}
are compact in  $H^{-1}_{loc}(\R \times \R^{+})$, for any $\eta^{0} \in C^{2}(\R)$.
 \end{lemma}
\begin{proof}
For simplicity,  we omit the
superscripts $\E, \delta$. From the estimate in (\ref{2.13}), for any function $\phi \in H_{0}^{1}$,  we have that
\begin{equation*}
| \int_0^{\infty} \int_{-\infty}^{\infty}
 \E v_{xx} \phi dxdt | = | \int_0^{\infty} \int_{-\infty}^{\infty}
 \E v_{x} \phi_{x} dxdt |,
\end{equation*}
is compact and the term $ \E (a_{x} w_{x}+ b_{x}z_{x}) $ on the right-hand side of (\ref{3.10}) is uniformly bounded,
if we choose $ \E $ to be much smaller than $ \delta $. Thus (\ref{3.12})  is compact in $H^{-1}_{loc}(\R \times \R^{+})$.
Similarly, we can prove from (\ref{2.11}) that (\ref{3.13})  is compact in $H^{-1}_{loc}(\R \times \R^{+})$.

To prove (\ref{3.14})  to be compact in $H^{-1}_{loc}(\R \times \R^{+})$, we multiply  (\ref{3.10}) by $\eta^{0}(v,u)_{v} $
and (\ref{3.11}) by  $ \eta^{0}(v,u)_{u}$
to obtain that
\begin{equation}\begin{array}{lcl}
\label{3.16} & & \eta^{0}(v^{\E,\delta},u^{\E,\delta})_{t}+ q^{0}(v^{\E,\delta},
u^{\E,\delta})_{x} \\\\
&=& \E \eta^{0}(v^{\E,\delta},u^{\E,\delta})_{xx} -   \E (a_{x} w_{x}+ b_{x}z_{x}) \eta^{0}(v^{\E,\delta},u^{\E,\delta})_{v}  \\\\
&-& \E(\eta^{0}(v^{\E,\delta},u^{\E,\delta})_{vv}
(v^{\E,\delta}_{x})^{2} +2 \eta^{0}(v^{\E,\delta},u^{\E,\delta})_{vu} v^{\E,\delta}_{x}u^{\E,\delta}_{x}+
\eta^{0}(v^{\E,\delta},u^{\E,\delta})_{uu}(u^{\E,\delta}_{x})^{2}).
\end{array}
\end{equation}
Since the boundedness estimates given in (\ref{2.11})-(\ref{2.13}), and for any $  \eta^{0} \in C^{2}(\R)$,  we know that the terms in (\ref{3.16}) satisfy
\begin{equation*}
|(a_{x} w_{x}+ b_{x}z_{x}) \eta^{0}(v^{\E,\delta},u^{\E,\delta})_{v}| \leq M(\delta)
\end{equation*}
and
\begin{equation*}
|\eta^{0}(v^{\E,\delta},u^{\E,\delta})_{vv}
(v^{\E,\delta}_{x})^{2} +2 \eta^{0}(v^{\E,\delta},u^{\E,\delta})_{vu} v^{\E,\delta}_{x}u^{\E,\delta}_{x}+
\eta^{0}(v^{\E,\delta},u^{\E,\delta})_{uu}(u^{\E,\delta}_{x})^{2} |\leq M(\delta),
\end{equation*}
where $ M(\delta) $ is a positive constant, which could tend to infinity as $ \delta $ tends to zero.

Moreover,  for any function $\phi \in H_{0}^{1} (\R)$,  we have that
\begin{equation*}
| \int_0^{\infty} \int_{-\infty}^{\infty}
  \eta^{0}(v^{\E,\delta},u^{\E,\delta})_{xx} \phi dxdt | = | \int_0^{\infty} \int_{-\infty}^{\infty}
 \eta^{0}(v^{\E,\delta},u^{\E,\delta})_{x} \phi_{x} dxdt | \leq M(\delta),
\end{equation*}
thus the right-hand side of (\ref{3.16}) is compact in  $H^{-1}_{loc}(\R \times \R^{+})$ if we choose $ \E $ to be much smaller than $ \delta $.
Lemma \ref{l2} is proved.
\end{proof}
Now we prove the Part I in Theorem \ref{main1}. From the definition (\ref{1.6}) of the generalized
 solutions of the Cauchy problem (\ref{1.3}) and (\ref{1.4}), and the equations (\ref{3.10}) and (\ref{3.11}), it is sufficient to prove
 the pointwise convergence of $(v^{\E,\delta},u^{\E,\delta})$. By using
the theory of compensated compactness \cite{Mu,Ta}, if  $ \nu_{x,t}$ is the family of positive probability measures
with respect to the viscosity solutions $(v^{\E,\delta},u^{\E,\delta})$ , we only need to prove that the support set of the Young measure $\nu_{x,t}$ is
concentrated on one point$(x,t)$, or the Young measure $\nu_{x,t}$ is a Dirac
measure.

Suppose the support set of the Young measure $\nu_{x,t}$ is
concentrated on the line $v=0$, or $ \mbox{ supp\,}
\nu_{x,t}=\{v=0\}$,  then since (\ref{3.12}) and (\ref{3.13}) given in Lemma \ref{l2},  using the measure equation to the
entropy-entropy flux pairs $(v,-u)$ and $(u, cv^{s})$, we get
\begin{equation*}
 <\nu_{x,t}, u>^{2}=<\nu_{x,t}, u^{2}>,
\end{equation*}
which inplies that $\nu_{x,t}$ is a Dirac measure and the support
set is one point $(0,\bar{u})$.

Suppose, for fixed $(x,t)$, the support set of the Young
measure $\nu_{x,t}$ is concentrated on $v \geq 0$ , but not only on $v=0$.
 Then clearly $<\nu_{x,t},
\eta^{0}> \ne 0$. As done in \cite{LPS,LPT}, using the measure equation in the theory of
compensated compactness to the entropy-entropy flux pairs
$(\eta^{0}(v,u),q^{0}(v,u))$, we get
\begin{equation*}\begin{array}{ll}
 & <\nu_{x,t},  \int_{-\infty}^{\infty}
f_{1}(\xi_{1})G(v,\xi_{1}-u) d \xi_{1}><\nu_{x,t},\int_{-\infty}^{\infty}
f_{2}(\xi_{2}) \theta \F{\xi_{2}-u}{v} G(v,\xi_{2}-u) d \xi_{2}>
\\\\& - <\nu_{x,t},  \int_{-\infty}^{\infty}
f_{2}(\xi_{2})G(v,\xi_{2}-u) d \xi_{2}><\nu_{x,t},\int_{-\infty}^{\infty}
f_{1}(\xi_{1}) \theta \F{\xi_{1}-u}{v} G(v,\xi_{1}-u) d \xi_{1}> \\\\
&=<\nu_{x,t}, \int_{-\infty}^{\infty} \int_{-\infty}^{\infty} f_{1}(\xi_{1}) f_{2}(\xi_{2}) \theta \F{\xi_{2}-\xi_{1}}{v} G(v,\xi_{1}-u) G(v,\xi_{2}-u) d \xi_{1} d \xi_{2}   >
\end{array}
\end{equation*}
or
\begin{equation}\begin{array}{ll}
\label{3.22}
& \int_{-\infty}^{\infty} f_{1}(\xi_{1})  <\nu_{x,t},
G(v,\xi_{1}-u) >  d \xi_{1}  \cdot  \int_{-\infty}^{\infty} f_{2}(\xi_{2})  <\nu_{x,t},  \F{\xi_{2}-u}{v} G(v,\xi_{2}-u)>  d \xi_{2}
\\\\& -  \int_{-\infty}^{\infty} f_{2}(\xi_{2})  <\nu_{x,t},
G(v,\xi_{2}-u) >  d \xi_{2} \cdot  \int_{-\infty}^{\infty} f_{1}(\xi_{1})  <\nu_{x,t},  \F{\xi_{1}-u}{v} G(v,\xi_{1}-u)>  d \xi_{1}
 \\\\
&= \int_{-\infty}^{\infty} \int_{-\infty}^{\infty} f_{1}(\xi_{1}) f_{2}(\xi_{2}) <\nu_{x,t},  \F{\xi_{2}-\xi_{1}}{v} G(v,\xi_{1}-u) G(v,\xi_{2}-u) > d \xi_{1} d \xi_{2}.
\end{array}
\end{equation}
(\ref{3.22}) holds for arbitrary functions $ f_{1}, f_{2}$, and this yields
\begin{equation}\begin{array}{ll}
\label{3.23} & <v H(v,u,\xi_{1})><
(\xi_{2}-u)H(v,u,\xi_{2})>
\\\\& -<v H(v,u,\xi_{2})><
(\xi_{1}-u)H(v,u,\xi_{1})>\\\\
&=<(\xi_{2}-\xi_{1})v H(v,u,\xi_{1})H(v,u,\xi_{2})>,
\end{array}
\end{equation}
where we use the notation $<H(v,u,\xi)>=<\nu_{x,t},H(v,u,\xi) >$
and
\begin{equation*}
H(v,u,\xi)=(v^{s+1}-(\xi-u)^{2})_{+}^{ \lam}.
\end{equation*}
Let $I=[w,z]$ for
each $(v,u) \in \mbox{ supp\,} \nu_{x,t}$, where $w=u-v^{\F{s+1}{2}}, z=u+v^{\F{s+1}{2}}$.  Dividing (\ref{3.23})
by $<v H(v,u,\xi_{1})><vH(v,u,\xi_{2})>$ and sending $\xi_{2}$ to $\xi_{1}$,
we obtain
\begin{equation}
\label{3.25} \F{ \partial}{ \partial \xi} \bigl( \F{< (\xi-u) H(v,u,\xi)>}{<vH(v,u,\xi)>} \bigr) =
\F{<v H(v,u,\xi)^{2}>}{<vH(v,u,\xi)>^{2}}.
\end{equation}
Again using the measure equation between
$(\eta^{0}(v,u),q^{0}(v,u))$ and $(v,-u)$, we get
\begin{equation*}\begin{array}{ll}
 &- \theta  <v>< (\xi-u)
H(v,u,\xi)>+<u><vH(v,u,\xi)>\\\\
&=- \theta  <(\xi-u)vH(v,u,\xi)>+<uvH(v,u,\xi)>,
\end{array}
\end{equation*}
which can be rewritten as
\begin{equation}\begin{array}{ll}
\label{3.27} &- \F{\theta <v><(\xi-u)
H(v,u,\xi)>+<u><vH(v,u,\xi)>}{<vH(v,u,\xi)>} \\\\
&=- \theta \xi +(1+\theta) \F{<uvH(v,u,\xi)>}{<vH(v,u,\xi)>}
\quad  \mbox{ in } I.
\end{array}
\end{equation}
Differentiating (\ref{3.27}) in $\xi$ and combining the outcome
with (\ref{3.25}), we also obtain
\begin{equation*}
\F{ \partial}{ \partial \xi} \bigl(
\F{<uvH(v,u,\xi)>}{<vH(v,u,\xi)>} \bigr) =-
\F{\theta}{1+\theta} \bigl(<v>
\F{<vH(v,u,\xi)^{2}>}{<vH(v,u,\xi)>^{2}}-1 \bigr) \leq 0.
\end{equation*}
 Following the steps given in Proposition II.1 in \cite{LPS} and Lemma 6 in \cite{LPT}, we can prove that the Young measure $\nu_{x,t}$ is a Dirac
measure, which implies the pointwise convergence of the viscosity solutions $(v^{\E,\delta},u^{\E,\delta})$. It is clear that
the limit $(v,u)$ of $(v^{\E,\delta},u^{\E,\delta})$ is a weak solution of the Cauchy problem (\ref{1.3}) and (\ref{1.4}) if we let
$ \E, \delta $ in (\ref{2.2}),(\ref{3.10}),(\ref{3.11}) go to zero. The Part I in Theorem 1 is proved.

Under the conditions given in the Part II in Theorem 1, we consider
the Cauchy problem (\ref{1.3}) and (\ref{1.4}), where
\begin{equation*}
u_{0}(x)= \int_{-\infty}^{x} v_{1}(\xi) d \xi.
\end{equation*}
Then from the condition (\ref{1.12}), we know $z'_{0}(x) \leq 0, w'_{0}(x) \leq 0 $ and so both $z_0(x)$ and $w_0(x)$  are decreasing.
From the conditions (\ref{1.11}) and (\ref{1.12}), we have that
 \begin{equation*}
(u_{0}( + \infty) + v_{0}^{ \theta}( + \infty))- (u_{0}(x) + v_{0}^{ \theta}(x))  = \int_{x}^{ + \infty} u'_{0}(\xi) + (v_{0}^{ \theta}(\xi))' d \xi \leq 0
\end{equation*}
and so
\begin{equation}\begin{array}{ll}
\label{3.31}
& z_0(x) = u_{0}(x)+  v_{0}^{ \theta}(x) \geq  u_{0}(+ \infty) + v_{0}^{ \theta}( + \infty) \\\\
&= \int_{ - \infty}^{ + \infty} v_{1}(\xi) d \xi + v_{0}^{ \theta}( + \infty) \geq -v_{0}^{ \theta}( - \infty),
\end{array}
\end{equation}
where the constant $ -v_{0}^{ \theta}( - \infty) $ is corresponding to  $ c_{0}$ in (\ref{1.10}).
Similarly
 \begin{equation*}
(u_{0}(x) - v_{0}^{ \theta}(x))- (u_{0}( - \infty) - v_{0}^{ \theta}( - \infty))   = \int_{- \infty}^{x} u'_{0}(\xi) - (v_{0}^{ \theta}(\xi))' d \xi \leq 0,
\end{equation*}
and so
\begin{equation*}
w_0(x) = u_{0}(x)-  v_{0}^{ \theta}(x) \leq  u_{0}(- \infty) - v_{0}^{ \theta}( - \infty) = - v_{0}^{ \theta}( - \infty).
\end{equation*}
Thus all conditions in Part I of Theorem 1 are satisfied. If we specially choose
$\phi_{1}=\phi_{t}, \phi_{2}=\phi_{x}$ and $v_{1}(x)=u_{0}'(x)$  in
(\ref{1.6}), the function $v$ given in Part I is a weak solution of the Cauchy problem (\ref{1.1}) and (\ref{1.2}), Theorem 1 is proved.

\section{Proof of Theorem \ref{main2}}

First, we prove  Part I in Theorem \ref{main2}.
We note that $w_0 , z_0 \in C(\R) \subset W^{1,1} _{loc} (\R)$.
Since $w_0$ and $z_0$ are decreasing, \eqref{xy} implies that
\begin{eqnarray*}
\lim_{x \rightarrow -\infty} w_0 (x) \geq w_0 (x) >  z_0 (y) \geq \lim_{x \rightarrow \infty} z_(x).
\end{eqnarray*}
Hence there exists a large constant $M_0 >0$ such that if $M \geq M_0$, then
\begin{eqnarray*}
w_0 (-M) > z_0 (M).
\end{eqnarray*} 

Now we prepare some lemmas.
\begin{lemma} \label{eqeq}
Suppose that the assumptions \eqref{con-1}-\eqref{con3} are satisfied. Then 
\begin{eqnarray} 
\left\{  \begin{array}{ll} \label{eqwz}
w_t + \lambda_2 w_x =0,\\
z_t + \lambda_1 z_x =0
\end{array} \right.  
\end{eqnarray} 
is satisfied for a.a.  $(x,t) \in \R \times \R^{+}$.
\end{lemma}
\begin{proof}
From \eqref{con--1} and \eqref{con--1-2}, since $v \in L^\infty (\R \times \R_+ )$ and $\theta=(s+1)/2 >1$, $\partial_x v^{\theta}$ can be defined in $L^1 _{loc} (\R \times \R^{+})$. From the second equation in \eqref{1.6} and the integration by parts, we have
\begin{eqnarray} \label{eq1}
\int_0 ^\infty \int_{-\infty} ^\infty  u_t \phi_2 -\theta  v^{\theta -1}  \partial_x ( v^{\theta}) \phi_2 dxdt=0.
\end{eqnarray}
Hence 
\begin{eqnarray} \label{rea1}
u_t  -\theta  v^{\theta -1}  \partial_x ( v^{\theta})=0
\end{eqnarray}
 is satisfied for a.a.  $(x,t) \in \R \times \R^{+}$.
Similarly, we have that $v_t -u_x =0$ is  satisfied for a.a.  $(x,t) \in   \R \times \R^+$. By multiplying this equation equality by $\theta v^{\theta -1}$, we have
\begin{eqnarray} \label{rea2}
\theta v^{\theta -1} v_t  -  \theta v^{\theta -1} \partial_x u =0
\end{eqnarray}
From $\eqref{rea1} + \eqref{rea2}$ and $\eqref{rea1} - \eqref{rea2}$, we have the first and the second equation in \eqref{eqwz} respectively.
\end{proof}
From \eqref{con0} and \eqref{con2}, $\lim_{x\rightarrow \pm \infty} (u_0 (x), v_0 (x))$ exists.
We put
$$\lim_{x\rightarrow \pm \infty} (u_0 (x), v_0 (x)) = (u_\pm, v_\pm).$$

\begin{lemma} \label{liml}
Suppose that the assumptions \eqref{con-1}-\eqref{con3} is satisfied, then we have
\begin{eqnarray} \label{lim}
\lim_{x\rightarrow \pm \infty} u(x,t) = u_\pm \ \mbox{and} \ \lim_{x\rightarrow \pm \infty} v(x,t) = v_\pm.
\end{eqnarray}
\end{lemma}
\begin{proof}
 Since $w$ and $z$ are decreasing with $x$, from the definition of $w$ and $z$, we have $u_x \leq 0$. So, from $ v_t - u_x =0  $, we have that
$v(x,t)$ is decreasing with $t$ for a.a. $x\in \R$. Hence $ 0 \leq v (x,t) \leq v_0 (x) \leq  \|v_0 \|_{L^\infty}$.  We put $\lambda_M = \theta   \|v_0 \|_{L^\infty}^{\frac{s-1}{2}}$.
Since $w_x, z_x \leq 0$ and $  0 \leq v \leq C_M$, by \eqref{eqwz}, we have 
\begin{eqnarray}
0 \leq  w_t  \leq -\lambda_M w_x \ \mbox{and} \
 \lambda_M z_x \leq z_t  \leq 0. \label{wzin}
\end{eqnarray}
We set $\rho_{j,\varepsilon}=\varepsilon^{-1} \rho_j (\cdot /\varepsilon )$  as standard mollifiers  with $x$ and $t$  for $j=1$ and $2$ respectively ($\rho_j \in C^\infty _0 (\R)$ and $\rho_j \geq 0$ and $\int_{-\infty} ^\infty  \rho_j (\cdot ) dx =1 $ for $j=1, 2$).
We note that
$$
\int_0 ^\infty \rho_{2,\varepsilon } (t-s) w_s (x,s) ds =\partial_t \int_0 ^\infty \rho_{2,\varepsilon } (t-s) w_s (x,s) ds + \rho_{2,\varepsilon } (t) w (x,0) .
$$
Hence, applying the mollifiers to the both side of the first inequality in \eqref{wzin}, we have
\begin{eqnarray*}
 {w_\varepsilon }_t  -\lambda_M {w_\varepsilon }_x +\rho_{2,\varepsilon } (t) \rho_{1,\varepsilon } * w_0 (x)  \leq 0 ,
\end{eqnarray*}
where $w_\varepsilon =\int_0 ^{\infty} \rho_{2,\varepsilon } (t-s) \rho_{1,\varepsilon } * w (x,s) ds$.
Noting $w_\varepsilon  (x+\lambda_M t,t)$ is differentiable with $t$, we have
\begin{eqnarray*} 
\dfrac{d}{dt}w_\varepsilon  (x+\lambda_M t, t) +\rho_{2,\varepsilon } (t) \rho_{1,\varepsilon } * w_0  (x+\lambda_M t) \leq 0.
\end{eqnarray*}
Integrating  on $[0,t]$, we have
\begin{eqnarray*}
w_\varepsilon  (x+\lambda_M t, t) - w_\varepsilon  (x,0) + \int_0 ^t \rho_{2,\varepsilon } (s) \rho_{1,\varepsilon } * w_0 (x+\lambda_M s) ds \leq 0.
\end{eqnarray*}
Since $w (x, \cdot)$ and $w(\cdot, t )$ are continuous with a.a. fixed $t$ and $x$ respectively, we have
$$
\lim_{\varepsilon \rightarrow 0} w_\varepsilon  (x,0) \rightarrow \int_{-\infty} ^0 \rho_2 (t) dt w_0 (x) 
$$ 
and 
$$
\int_0 ^t \rho_{2,\varepsilon } (s) \rho_{1,\varepsilon } * w (x+\lambda_M s, 0) ds \rightarrow \int_{0} ^\infty \rho_2 (t) dt w_0 (x).
$$
Hence we have by taking $\varepsilon \rightarrow 0$,
\begin{eqnarray*}
w  (x+\lambda_M t,t) - w_0 (x)\leq 0.
\end{eqnarray*}
Therefore we have 
$$
w_0 (x) \leq  w (x,t)  \leq w (x-\lambda_M t, 0)
$$
and
$$
z_0 (x+ \lambda_M  t)   \leq z (x,t)  \leq z_0 (x),
$$
which implies that \eqref{lim}.
\end{proof}
We put 
\begin{eqnarray*}
F(t)=-\int_{-\infty} ^\infty  v (x,t)- v_0 (x)dx.
\end{eqnarray*}
From the first equation \eqref{1.3} and Lemma \ref{liml},  we have
\begin{eqnarray}\label{ff}
F(t)= (u_- - u_+)t.
\end{eqnarray}
We divide $F(t)$ into the three parts as follows:
\begin{align*}
F(t)&=\left(\int_{-\infty} ^{-M} +\int_{-M} ^{M} + \int_{M} ^\infty   \right) v (x,t)- v_0 (x) dx \\
&=F_1 (t) + F_2 (t) + F_3 (t).
\end{align*}
Now we estimate $F_1 (t)$. From the first equation in \eqref{1.3} and Lemma \ref{liml}, we have
\begin{eqnarray*}
\dfrac{d}{dt} F_1 (t) =\int_{-\infty} ^{-M} -u_x (x,t) dx =-u(t,-M)+ u_-. 
\end{eqnarray*}
Since $w_0 (-M) \leq w (-M,t)$, from the definition of $w$, we have
$$
-u (M,t) \leq -v^\theta (-M,t) - w_0 (-M) \leq -w_0 (-M) 
$$
Hence we have
\begin{eqnarray}\label{f1}
F_1 (t) \leq -t (u_- - w_0 (-M)).
\end{eqnarray}
Since $v\geq 0$ under our contradiction argument and $v \leq C$, we can estimate $F_2 (t)$ as
\begin{eqnarray} \label{f2}
F_2 (t) \leq C_M,
\end{eqnarray}
where $C_M$ is a positive constant depending on $M$.
In the same way as in the above estimate of $F_1$, we have
\begin{eqnarray} \label{f3}
F_3 (t) \leq t (-u_+ +z_0 (M)).
\end{eqnarray}
From \eqref{ff}, \eqref{f1}, \eqref{f2} and \eqref{f3}, we have
\begin{eqnarray*}
C_M + (z_0 (M) - w_0 (-M) + u_- - u_+)t \geq (u_- - u_+)t .
\end{eqnarray*}
Hence  we have 
$$
C_M >(w_0 (-M)- z_0 (M))t ,
$$
which gives a contradiction for large $t$, since $w_0 (-M)- z_0 (M)>0$.
Therefore we complete the proof of Part I in Theorem \ref{main2}.

Next we prove Part II in Theorem \ref{main2}.
For the solutions of \eqref{1.1}, we put $u=\int_{-\infty} ^x v_t dx$ and $\tilde{u}=-\int_x ^\infty v_t dx$. 
We note $u$ is well-defined, if we assume \eqref{con4}.
We show the following Lemma.
\begin{Lemma}
Let $v$ be a weak solutions of the generalized Cauchy problem  \eqref{1.5} and $u$ and $\tilde{u}$ be as above.
Suppose that \eqref{con-1}-\eqref{con4} are satisfied. Then
$(v,u)$  and $(v,\tilde{u})$ are solutions of the generalized Cauchy problem  \eqref{1.6}.
\end{Lemma}
\begin{proof}
We put $\psi \in C^\infty _0 (\R)$ such that $\psi (0) =1$. Replacing $\phi$ by $\psi (\varepsilon x) \int_x ^{\infty} \phi  (t,y) dy$  in the generalized Cauchy problem \eqref{1.5}, from \eqref{con2}, \eqref{con3} and the integration by parts, we have that
\begin{align*}
&\int_0 ^\infty u(x,t)  \psi(\varepsilon y)  \partial_t \phi (t,y) -c  (v^s)  \partial^2 _x \left(\psi (\varepsilon x) \int_x ^{\infty} \phi  (t,y) dy\right) dx dt \\
&-\int_{-\infty} ^\infty  u_0 (x)   \psi(\varepsilon y)   \phi (0,y) dy=0.
\end{align*}
Taking $\varepsilon \rightarrow 0$,  from \eqref{con2}, we have 
\begin{align*}
\int_0 ^\infty \int_{-\infty} ^\infty u (x,t)   \partial_t \phi (t,y) -c  (v^s)  \partial _x  \phi  (t,y) dy dx dt -\int_{-\infty} ^\infty  u_0 (x)     \phi (0,y) dy =0.
\end{align*}
Similarly,  replacing $\phi$ by $\psi (\varepsilon x) \int_{-\infty} ^{x} \phi  (t,y) dy$, we obtain that
\begin{align*}
\int_0 ^\infty \int_{-\infty} ^\infty  \tilde{u} (x,t)   \partial_t \phi (t,y) - c  (v^s)  \partial_x  \phi  (t,y) dy dx dt -\int_{-\infty} ^\infty  \tilde{u_0} (x)     \phi (0,y) dy =0.
\end{align*}
\end{proof}
For the Riemann invariant \eqref{1.9-2}, we obtain same results as in Lemmas \ref{eqeq} and \ref{liml}.
Therefore we can show Part II in Theorem \ref{main2} by using similar contradiction argument to in the proof of Part I.

\begin{Remark}
Here we show that the solution constructed in Theorem 1 satisfies \eqref{con--1-2}, \eqref{con2} and \eqref{con3},
if $w_0 (x)$ and $z_0 (x)$ are  bounded, decreasing and in $W^{1.1} _{loc} (\R)$.
Since the approximate solution $(R,S)=(w_x, z_x)$ to \eqref{2.8} satisfies that
$$
R_t + (\lambda_2 R)_x =R_{xx} \ \mbox{and} \ S_t + (\lambda_2 S)_x =S_{xx},
$$
we have from the negativity of $R$ and $S$
\begin{align*}
\| w_x (t)\|_{L^1}+ \| z_x (t) \|_{L^1} & = -\int_{-\infty} ^\infty  R(x,t) +S(x,t)  dx  \\
& = -\int_{-\infty} ^\infty  R(x,0) +S(x,0)  dx \\
& =\lim_{x\rightarrow \infty}2(u_0 (-x) - u_0 (x)).
\end{align*}
The $\lim_{x\rightarrow \infty}2(u(-x) - u(x))$ exists and is finite, since $u_0 =(w_0 +z_0 )/2$  is decreasing and bounded.
Hence, for solutions of  the non-viscous equations \eqref{1.3},  we have that $w(t,\cdot) , z(t,\cdot) \in W^{1,1} _{loc} (\R)$ for a.a. $t \geq 0$
 and that
\eqref{con2} is satisfied.
Furthermore, in the similar way as to the proof of Lemma \ref{eqeq}, we have that $w , z \in W^{1,1} _{loc} (\R \times \R^{+})$ 
and \eqref{eqwz} are satisfied. Therefore we have by the boundedness of $\lambda_1$ and $\lambda_2$ that
\begin{align*}
\| w_t (t)\|_{L^1}+ \| z_t (t) \|_{L^1} & \leq \| \lambda_2 w_x (t)\|_{L^1}+ \|  \lambda_2 z_x (t) \|_{L^1}\\
 & \leq  C( \|  w_x (0)\|_{L^1}+ \|   z_x (0) \|_{L^1}).
\end{align*}

\end{Remark}





\noindent {\bf Acknowledgments:}  This paper of the second author is partially supported by a Qianjiang professorship of Zhejiang Province of China,
the National Natural Science Foundation of China (Grant No.  11271105). The second author is supported by  Grant-in-Aid for Young Scientists Research (B), No. 16K17631.

\end{document}